\documentclass[a4paper]{article}

\usepackage[latin1]{inputenc}
\usepackage[T1]{fontenc}
\usepackage[english]{babel}

\usepackage{amsmath}
\usepackage{amssymb}
\usepackage{amsfonts,makeidx,amsthm,mathrsfs}
\usepackage{amscd}

\makeatletter
\renewcommand\@biblabel[1]{#1.}
\makeatother

\newcommand{\Fcal}{\mathcal{F}}

\newcommand{\Lr}{\mathscr{L}}

\newcommand{\U}{\mathcal{U}}

\newcommand{\W}{\mathcal{W}}
\newcommand{\D}{\mathcal{D}}
\newcommand{\R}{\mathbb{R}}

\newcommand{\Z}{\mathbb{Z}}
\def\qInn{\mathop{\mathrm{qInn}}\nolimits}

\def\fl{\mathop{\mathrm{Flux}}\nolimits}

\def\ham{\mathop{\mathrm{Ham}}\nolimits}
\def\Ham{\mathop{\mathrm{\mathfrak{H}am}}\nolimits}
\def\symp{\mathop{\mathrm{Symp}}\nolimits}
\def\Symp{\mathop{\mathrm{\mathfrak{S}ymp}}\nolimits}

\def\aut{\mathop{\mathrm{Aut}}\nolimits}
\def\der{\mathop{\mathrm{Der}}\nolimits}
\def\cl{\mathop{\mathrm{Cl}}\nolimits}

\newtheorem{theoremintro}{Theorem}
\newtheorem{theoremprinc}{Theorem}
\newtheorem{theorem}{Theorem}[section]
\newtheorem{lemme}[theorem]{Lemma}
\newtheorem{cor}[theorem]{Corollary}
\newtheorem{prop}[theorem]{Proposition}
\newtheorem{defi}[theorem]{Definition}
\newtheorem{quest}[theorem]{Question}
\theoremstyle{definition} 
\newtheorem{ex}[theorem]{Example}

\theoremstyle{remark}

\newtheorem{rem}[theorem]{Remark}

\begin{document}

\title{The group of Hamiltonian automorphisms of a star product}

\author{\small
Laurent La Fuente-Gravy \\
\scriptsize{laurent.lafuente@uclouvain.be}\\
\footnotesize{Institut de Recherche en Math\'ematique et Physique, Universit\'e catholique de Louvain}\\[-7pt]
\footnotesize Chemin du cyclotron, 2, 1348 Louvain-la-Neuve, Belgium\\[-7pt]}

\maketitle


\begin{abstract}
We deform the group of Hamiltonian diffeomorphisms into a group of Hamiltonian automorphisms, $\ham(M,*)$, of a formal star product $*$
on a symplectic manifold $(M,\omega)$. We study the geometry of that group and deform the Flux morphism in the framework
of deformation quantization.
\end{abstract}

\noindent {\footnotesize {\bf Keywords:} Deformation quantization, Automorphisms of star product, Flux morphism, Hamiltonian automorphisms group.\\
{\bf Mathematics Subject Classification (2010):}  Primary 53D55, Secondary 22E65, 57N20 }

\tableofcontents


\section{Introduction}

We define the group of Hamiltonian automorphisms of a star product. 
It is a deformation of the group of Hamiltonian diffeomorphisms of a symplectic manifold in the framework of formal deformation quantization. We also study the geometric properties of this group.

The group $\ham(M,\omega)$ of Hamiltonian diffeomorphisms is a normal subgroup of $\symp_0(M,\omega)$ the
connected component of the group of all symplectomorphisms of a symplectic manifold $(M,\omega)$.
Banyaga \cite{ban2} showed that the group $\ham(M,\omega)$ is the kernel of a morphism defined
on $\symp_0(M,\omega)$ with values in $H^1_{dR,c}(M)/\Gamma(M,\omega)$ where $H^1_{dR,c}(M)$ is the first
de Rham cohomology group with compact support and $\Gamma(M,\omega)$ is the so-called Flux group. He used this 
characterisation to prove that the group $\ham(M,\omega)$ is simple when the manifold is closed, extending the results
of Thurston on volume preserving diffeomorphisms to the symplectic case. Observing that the group $\Gamma(M,\omega)$
is at most countable, he proved that the Lie algebra of $\ham(M,\omega)$ is the space of compactly supported Hamiltonian vector fields.
In 2006, Ono \cite{Ono} showed that $\Gamma(M,\omega)$ is discrete when the manifold $(M,\omega)$ is closed; this proved the famous Flux conjecture which states that $\ham(M,\omega)$ is $C^1$-closed in $\symp_0(M,\omega)$. 

A similar approach in the framework of deformation quantization ($*$ product) on a symplectic manifold leads as a first step
to the study of the group of Hamiltonian automorphisms of a star product. 

To avoid technical difficulties we will assume throughout the paper that $(M,\omega)$ is a closed symplectic manifold. Let $*$ be a star product on $(M,\omega)$. Hamiltonian automorphisms are the solutions of the Heisenberg equation
\begin{equation}
\frac{d}{dt}A_t^H = D_{H_t}A_t^H := \frac{1}{\nu} [H_t,A_t^H . ]_*, \textrm{ with initial condition } A_0^H=Id,
\end{equation}
where $D_{H_t}$ is a smooth family of quasi-inner derivations. We then set 
\begin{equation}
\ham (M,*):=\{A=A_1^H \textrm{ for such } D_{H_t}\}.
\end{equation}
Our first observation is that $\ham(M,*)$ is a normal subgroup of $\aut_0(M,*)$ the group of automorphisms
of the star product deforming $\symp_0(M,\omega)$. The group $\ham(M,*)$ comes with an anti-epimorphism
$\cl : \ham(M,*) \rightarrow \ham(M,\omega)$. 

We prove that $\ham(M,*)$ is the kernel of a morphism defined on $\aut_0(M,*)$.
More precisely, we define a formal version of the flux morphism denoted by $\fl^*$ and we obtain : 
\begin{theoremintro} \label{theor:SES}
There is a short exact sequence of groups
$$1\rightarrow \ham(M,*) \rightarrow \aut_0(M,*) \stackrel{\Fcal}{\rightarrow} \frac{H^1_{dR}(M)[[\nu]]}{\Gamma(M,*)} \rightarrow 1,$$
where $\Fcal(A):=\fl^*(\{A_t\})$ for any smooth path in $\aut_0(M,*)$ joining $A$ to the identity and $\Gamma(M,*):=\fl^*(\pi_1(\aut_0(M,*)))$ where $\pi_1(\aut_0(M,*))$ is the subgroup of $\widetilde{\aut}_0(M,*)$ consisting of classes of smooth loops of automorphisms.
\end{theoremintro}

Next, we observe that $\Gamma(M,*)$ is the image of a morphism
$$\fl^*_{def}:\pi_1(\symp_0(M,\omega)) \rightarrow H^1_{dR}(M)[[\nu]].$$
The values of $\fl^*_{def}$
only depend on the equivalence class of the star product.
We give a condition on the group $\fl^*_{def}(\pi_1(\ham(M,\omega)))$ which ensures that the Lie algebra of $\ham(M,*)$ is the space of quasi-inner derivations. We gives examples of this situation.

In section \ref{sect:Fed}, using the Fedosov construction of star product, we give an explicit expression of the deformed
flux in terms of the characteristic $2$-form parametrizing the star product. This works on nice elements of $\pi_1(\symp_0(M,\omega))$. Consider $*_{\Omega,\nabla}$ a Fedosov's star product obtained with the help of a symplectic connection $\nabla$
and a series of closed $2$-forms $\Omega$, we obtain : 
\begin{theoremintro} \label{fl comp} 
Let $\{\varphi_t\}$ be a loop of symplectomorphisms generated by the symplectic vector field
$X_t$ such that $\varphi_t^*\Omega=\Omega$ and $\varphi_{t*}\nabla=\nabla$.
Then, the deformed flux of $\{\varphi_t\}$ defined with the star product $*_{\Omega,\nabla}$ is
\begin{equation}
\fl_{def}^{*_{\Omega,\nabla}}(\{\varphi_t\})=\int_0^1[i(X_t) \omega]dt - \left[\int_0^1 \varphi^*_t i(X_{t}) \Omega dt\right].
\end{equation}
\end{theoremintro}

\subsection*{Acknowledgement}
The results in this paper come form my Ph.D. Thesis \cite{LLF} and I warmly thank my advisors Simone Gutt, Fr\'ed\'eric Bourgeois and Michel Cahen. This work benefited from a grant from the Fond National de la Recherche Scientifique in Belgium.
It was also supported by an Action de Recherche Concentr\'ee of the Communaut\'e Fran\c caise de Belgique and the Interuniversity Attraction Pole ``Dynamics, Geometry and Statistical Physics" network.


\section{Star products, derivations and automorphisms} \label{sect:Recall}

In this section, we recall the definitions and basic properties of star products on symplectic manifolds that we need
in this paper.

Let $(M,\omega)$ be a symplectic manifold. The space $C^{\infty}(M)$ of real valued smooth functions
is naturally endowed with a Poisson bracket : 
$$\{F,G\}=-\omega(X_F,X_G),\ \forall F,G \in C^{\infty}(M) $$
where $X_F$ is a Hamiltonian vector field on $M$, that is $i(X_F)\omega:=dF$. 

A {\bf star product} on $(M,\omega)$ is a $\R[[\nu]]$-bilinear associative law on the space $C^{\infty}(M)[[\nu]]$ of formal series
of smooth functions : 
\begin{equation} \nonumber 
*: (C^{\infty}(M)[[\nu]] )^2 \rightarrow C^{\infty}(M)[[\nu]]: (H,K)\mapsto H*K:=\sum_{r=0}^{\infty} \nu^r C_r(H,K) 
\end{equation}
where the $C_r$'s are bidifferential operators null on constants such that for all $H,\, K \in C^{\infty}(M)[[\nu]]$ : 
$C_0(H,K)=HK$ and
$C_1(H,K)-C_1(K,H)=\{H,K\}$.
It is a result of De Wilde-Lecomte \cite{DWL}, Fedosov \cite{fed2} and Omori-Yoshioka-Maeda \cite{OMY} that all symplectic manifolds admit a star product.
Two star products $*$ and $*'$ are {\bf equivalent} if there exists a formal power series $T$ of $\R[[\nu]]$-linear differential operators
$$T=Id+\sum_{r=1}^{\infty} \nu^rT_r: C^{\infty}(M)[[\nu]] \rightarrow C^{\infty}(M)[[\nu]]$$
such that for all $F,G \in C^{\infty}(M)[[\nu]]$, we have : 
$T(F*G)=TF *' TG.$ Star products on symplectic manifolds are classified up to equivalence by  $H^2_{dR}(M)[[\nu]]$, see for example \cite{BCG}.

\subsection{Derivations} \label{subsect:deriv}

At the classical level, a derivation of the Poisson algebra of a symplectic manifold is a symplectic vector field, that is a vector field $X$ on $M$ such that $\Lr_X\omega=0$.

Now, we fix a star product $*$ on the symplectic manifold $(M,\omega)$. 
A {\bf  derivation} of the star product is a $\R[[\nu]]$-linear map $D:C^{\infty}(M)[[\nu]] \mapsto C^{\infty}(M)[[\nu]]$, such that
\begin{equation*}
D(F*G)=DF*G+F*DG.
\end{equation*}
We denote by $\der(M,*)$ the {\bf space of derivations}. It is a Lie algebra for the commutator of derivations.

 A  derivation is called {\bf quasi-inner} if it is of the form
$D_H(F):= \frac{1}{\nu}[H,F]_*$, for all $F\in C^{\infty}(M)[[\nu]],$
for some $H\in C^{\infty}(M)[[\nu]]$. We denote by $\qInn(M,*)$ the {\bf space of quasi-inner derivations}. It is an ideal of $\der(M,*)$.

Derivations of the star product are in bijection with formal series of symplectic vector fields on $M$, the last space will be denoted by $\Symp(M,\omega)[[\nu]]$: 
$$p:\der(M,*)\rightarrow \Symp(M,\omega)[[\nu]]: D \mapsto p(D),$$
such that on a contractible open set $U$ : 
$$ i(p(D))\omega\vert_U = dH_U \textrm{ with } DF\vert_U=\frac{1}{\nu}[H_U,F]_*.$$
for some $H_U \in C^{\infty}(U)[[\nu]]$.

\begin{prop} \label{prop:lienderaut}
Let $(M,\omega)$ be a symplectic manifold endowed with a star product $*$. Then,
\begin{enumerate}
\item $p(D)=D_0+ \nu (\ldots)$ where $D=\sum_{i=0}^{\infty}\nu^i D_i\in \der(M,*)$.
Moreover, $D=D_F\in \qInn(M,*)$ for $F\in C^{\infty}(M)[[\nu]]$ if and only if $p(D)=X_F$ with
$i(X_F)\omega=dF$ formally in $\nu$.
\item Let $D,D'\in \der(M,*)$, then $[D,D']\in \qInn(M,*)$.
\item If $D_F$ and $D_G$ are quasi-inner derivations, we have $[D_F,D_G]=D_{\frac{1}{\nu}[F,G]_*}$.
\end{enumerate}
\end{prop}

\subsection{Automorphisms}

From now on, we will assume that the symplectic manifold $(M,\omega)$ is closed. At the classical level, smooth one parameter families of symplectic vector fields integrate to families of symplectomorphisms. Recall that a symplectomorphism is a diffeomorphism $\phi:M\rightarrow M$ such that $\phi^*\omega=\omega$.

We indicate below how to ``exponentiate'' a family of derivations into a family of automorphisms. This is analogous
to what have been done in \cite{fed} and \cite{wa}. 

An {\bf automorphism} of the star product is a $\R[[\nu]]$-linear bijection $A:C^{\infty}(M)[[\nu]] \mapsto C^{\infty}(M)[[\nu]]$, such that
\begin{equation*}
A(F*G)=AF*AG.
\end{equation*}
The group $\aut(M,*)$ of all automorphisms of the star product projects onto the group of symplectomorphisms, denoted by $\symp(M,\omega)$. Indeed, if $A=\sum_{r=0}^{\infty} \nu^r A_r\in \aut(M,*)$, then $A_0:=\varphi^*$ for some $\varphi \in \symp(M,\omega)$.
Hence, the map {\bf classical limit} defined by  
$$\cl : \aut(M,*) \rightarrow \symp(M,\omega):A\mapsto \varphi$$
is an anti-homomorphism of group. If $A\in \cl^{-1}(Id)$, then there exists $D \in \nu \der(M,*)$ such that $A=\exp(D)$. 

\begin{defi} \label{def:aut0}
The {\bf subgroup $\aut_0(M,*)$} of $\aut(M,*)$, is defined to be $\cl^{-1}(\symp_0(M,\omega))$, where $\symp_0(M,\omega)$ is the identity component (for the compact-open $C^{\infty}$-topology) of the group of symplectic diffeomorphisms.
\end{defi}

\begin{defi} \label{smoothder}
Let $I$ be an interval in $\R$.
A one-parameter family of derivations $D_t=\sum_{r=0}^{\infty}\nu^r D_{r,t} \in \der_0(M,*)$ for $t\in I$
is called {\bf smooth} if for all $F\in C^{\infty}(M)$ we have $D_t(F)\in C^{\infty}(I \times M)[[\nu]]$,
\end{defi}

\begin{rem}
Using the bijection $p:\der(M,*)\rightarrow \Symp(M,\omega)[[\nu]]$ defined in the above Subsection \ref{subsect:deriv}, one sees that a one-parameter family of derivations $D_t$ is smooth if the coefficients of $p(D_t)$ are smooth one parameter families of symplectic vector fields, where $\Symp(M,\omega)$ is endowed with the compact-open $C^{\infty}$-topology.
\end{rem}

\begin{prop} \label{flow} 
Let $D_t=\sum_{r=0}^{\infty} \nu^rD_{r,t}$ be a smooth one-parameter family of derivations.
Then there exists a unique family of automorphisms $t\mapsto A_t $ defined for all $t$ such that
\begin{equation} \label{eq:Heis} 
\frac{d}{dt} A_t H= D_tA_t H,\ \forall H \in C^{\infty}(M)[[\nu]],
\end{equation}
with the initial condition $A_0H=H$ for all $H\in C^{\infty}(M)[[\nu]] $.

Moreover, if the derivation D does not depend on the time $t$, then the family $A_t$ is a one-parameter subgroup of automorphisms and $D\circ A_t=A_t \circ D$.

Finally, when $D_t \in \nu \der_0(M,*)$, then the solution $A_t$ of the equation (\ref{eq:Heis}) 
is the formal exponential $A_t=\exp(\int_0^t D_s ds)$, the integral $\int_0^t D_s ds$ is the derivation $F\mapsto \int_0^t D_s(F) ds$.
\end{prop}

\begin{proof}
Let $H \in C^{\infty}(M)[[\nu]]$, we will show that there exists a unique family $H(t)\in C^{\infty}(M)[[\nu]] $ such that 
\begin{equation} \label{eqheis} 
\frac{d}{dt}H(t)= D_t H(t)
\end{equation}
and $H(0)=H$.

Let $\varphi_t$ be the one-parameter family of symplectomorphisms generated by $-D_{t,0}$, the opposite of the zeroth order term of $D_t$. 
It means $\frac{d}{dt}\varphi_t=-\varphi_t^*D_{t,0}$ and $\varphi_0=Id$.
Then, if $H(t)$ satisfies (\ref{eqheis}), we have
\begin{equation}
\frac{d}{dt}\varphi_t^*H(t)  =  -\varphi_t^*D_{t,0}H(t) + \varphi_t^*D_tH(t)  =  \varphi_t^* \tilde{D}_tH(t).
\end{equation}
where $\tilde{D}_t:= D_t-D_{t,0}$. After integration with respect to $t$, we get
\begin{equation} \label{eqrecu} 
\varphi_t^*H(t) = H + \int_0^t \varphi_s^* \tilde{D}_sH(s)ds.
\end{equation}
Now the equation (\ref{eqrecu}) can be solved by induction on the degree in $\nu$ and the solution is unique.
Set $A_t: C^{\infty}(M)[[\nu]] \rightarrow C^{\infty}(M)[[\nu]] : H \mapsto H(t)$.

It remains to show that $A_t$ is an automorphism of star product. For this, consider the two expressions
$A_tH * A_tK$ and $A_t(H*K)$. They are equal for $t=0$ and are both solutions of equation (\ref{eqheis})
for all $t$.
Then, by uniqueness of the solutions of the equation (\ref{eqheis}), we have $A_tH*A_tK=A_t(H*K).$

The fact that $A_t$ is a one-parameter subgroup when the derivation is autonomous is again a consequence of the uniqueness of $A_t$ (see \cite{wa}). 

The last statement is checked by differentiating $\exp(\int_0^t D_s ds)H$ for $H\in C^{\infty}(M)[[\nu]]$.
\end{proof}

\begin{rem} \label{rem:-Dt0}
The solution $A_t$ of the equation (\ref{eq:Heis}) starts at order $0$ in $\nu$ by $(\varphi_t^{-1})^*$,
where $\varphi_t$ is the flow of $-D_{t,0}$. In general, $\varphi_t^{-1}$ is NOT the flow of $D_{t,0}$.
\end{rem}

\begin{defi} \label{def:smoothaut}
Let $I$ be an interval in $\R$.
A one-parameter family of automorphisms $A_t=\varphi_t^* + \sum_{r=1}^{\infty}\nu^r A_{r,t}$ for $t\in I$ is called smooth if 
for all $F\in C^{\infty}(M)$ we have $A_t(F)\in C^{\infty}(I \times M)[[\nu]]$,
\end{defi}

\begin{rem}
Using the Weinstein tubular neighbourhood, one defines a chart $W:U\subset \symp_0(M,\omega) \mapsto V \subset \Symp(M,\omega)$ from a neighbourhood $U$ of $Id$ in $\symp_0(M,\omega)$ and a neighbourhood $V$ of $0$ in $\Symp(M,\omega)$. Together with Proposition \ref{flow}, one sees that $\cl^{-1}(U)$ is in bijection with $W(U)\times \nu\der(M,*)$ or equivalently with $W(U)+\nu\Symp(M,\Omega)[[\nu]]$. A one-parameter family of automorphisms $A_t$ in $\cl^{-1}(U)$ is smooth if the coefficients of its image in $W(U)+\nu\Symp(M,\Omega)[[\nu]]$ are smooth one-parameter families of symplectic vector fields.
\end{rem}

\begin{cor}
If $D_t$ is a smooth one parameter family of derivation, then the solution $A_t$ of Equation (\ref{eq:Heis}) is smooth.
\end{cor}

Given $D_t$ a smooth family of derivations, we say that the solution path $A_t$ of the equation (\ref{eq:Heis}) is {\bf generated} by $D_t$. One has : 

\begin{prop}[Computation rules] \label{prop:rules} 
Let $A_t$ and $A'_t$ be smooth paths of automorphisms generated by $D_t$ and $D'_t \in \der(M,*)$.
Then,
\begin{enumerate}
\item the path $A_tA'_t$ is generated by the derivation $D_t+A_tD'_tA_t^{-1}$,
\item the path $A_t^{-1}$ is generated by the derivation $-A_t^{-1}D_tA_t$,
\item we have $A_tD'_tA_t^{-1}=D'_t + D_{F_t}$ for a family $F_t \in C^{\infty}(M)[[\nu]]$.
\item Let $D,D' \in \nu\der(M,*)$, then $\exp(D)\exp(D')=\exp(D+D'+D_F)$ for some $F\in C^{\infty}(M)[[\nu]]$.
\end{enumerate}
\end{prop}

\begin{proof}
To prove point $1$, it suffices to differentiate the path $A_tA'_t$. Indeed,
$$\frac{d}{dt}A_tA'_t= (D_t+ A_tD'_tA_t^{-1})\circ A_tA'_t.$$
Point $2$ is obtained by applying $1$ to the path $Id=A_tA_t^{-1}$. \\
For point $3$, we compute $A_tD'_t(A_t)^{-1}= D'_t + \int_0^t \frac{d}{ds} A_sD'_t(A_s)^{-1} ds.$
Applying point 2, we compute $\frac{d}{ds} A_sD'_t(A_s)^{-1} ds= [D_s, A_sD'_t(A_s)^{-1}]$.
Now, the commutator of two derivations is quasi-inner  (by Proposition \ref{prop:lienderaut}).\\
Point $4$ is obtained by applying points $1$ and $3$ to the path $\exp(tD)\exp(tD')$.
\end{proof}


\section{The group of Hamiltonian automorphisms} \label{sect:ham}

We integrate the Lie algebra $\qInn(M,*)$ of quasi-inner derivations to produce the group $\ham(M,*)$.

Consider smooth one-parameter families of derivations of the form
$$D_{H_t}:= \frac{1}{\nu}[H_t,.]_* \in \qInn(M,*).$$ 
By Proposition $\ref{flow}$ there 
exists a one-parameter family of automorphisms $A^H_t$ such that $\frac{d}{dt}A^H_t=D_{H_t}A^H_t$.
We say that $A^H_t$ is {\bf generated} by the time-dependent Hamiltonian $H_t$.

\begin{defi}
The set of {\bf Hamiltonian automorphisms} is the set
\begin{equation}
\ham(M,*):= \{ A\in \aut(M,*) \ \vert \  A=A^H_1 \textrm{ for such } D_{H_t} \in \qInn(M,*) \}.
\end{equation}
The map $\cl$ restricts to a surjection $\cl:\ham(M,*)\mapsto \ham(M,\omega)$.
\end{defi}

\begin{lemme} \label{lem:AHr}
For all $A\in \ham(M,*)$, $D_G\in \qInn(M,*)$ : 
$AD_GA^{-1}=D_{AG}$.
\end{lemme}

\begin{proof}
The proof is a direct computation.
\end{proof}

\begin{theorem} \label{theor:hamgp}
Let $*$ be a star product on a symplectic manifold $(M,\omega)$, then $\ham(M,*)$ is a normal subgroup of $\aut_0(M,*)$. 

There is an anti-epimorphism $\cl:\ham(M,*)\rightarrow \ham(M,\omega)$.
\end{theorem}

\begin{proof}
Let $A$, $B\in \ham(M,*)$, we show that $AB\in \ham(M,*)$. Write $A^H_t$ and $B^G_t$, the one-parameter families generated 
by $H_t$ and $G_t$ respectively such that $A=A^H_1$ and $B=B^G_1$.
Using Proposition \ref{prop:rules} and Lemma \ref{lem:AHr}, we see that $A^H_tB^G_t$ is generated by the Hamiltonian  $K_t:=H_t+A^H_tG_t$. So, $AB=A^H_1B^G_1$ is in $\ham(M,*)$.

Let $A\in \ham(M,*)$. Since $A=A_1^H$ for some $D_{H_t}\in \qInn(M,*)$, then $A^{-1}=(A^H_1)^{-1}$.
Using the computation rules \ref{prop:rules} and Lemma \ref{lem:AHr}, we know that $(A^H_t)^{-1}$ is generated by $-(A^H_t)^{-1}H_t$.

The fact that $\ham(M,*)$ is a normal subgroup of $\aut_0(M,*)$ is a consequence of the following identities. Let $A\in \aut_0(M,*)$ and $A_t^H$ the family of Hamiltonian
automorphisms generated by $H_t\in C^{\infty}(M)[[\nu]]$ then 
\begin{equation}
\frac{d}{dt}A A^H_t A^{-1}=AD_{H_t}A^H_tA^{-1}=D_{AH_t}A A^H_t A^{-1}.
\end{equation}

We immediately have that the projection $\cl$ is an anti-epimorphism of group. 
\end{proof}


\begin{prop} \label{equivalence} 
Let $*$ and $*'$ be two equivalent star products on $C^{\infty}(M)[[\nu]]$. Denote
by $T:(C^{\infty}(M)[[\nu]],*) \rightarrow (C^{\infty}(M)[[\nu]],*')$ an equivalence of
star product. Then the map
\begin{equation}
C_T:\ham(M,*) \rightarrow \ham(M,*') : A \mapsto TAT^{-1}
\end{equation}
is an isomorphism of group.
\end{prop}

\begin{proof}
Let $A_t^H\in \ham(M,*)$ generated by $H_t\in C^{\infty}(M)[[\nu]]$, then 
$C_T(A_t^H)=TA_t^HT^{-1}$ is generated by $TH_t$. So, $C_T(A_t^H)\in \ham(M,*')$.
The map $C_T$ is clearly invertible and it is a morphism of group.
\end{proof}


\section{The formal flux morphism} \label{sect:flux}

The goal of this section is to describe $\ham(M,*)$ 
as the kernel of a morphism on $\aut_0(M,*)$; as in the classical case, where the group of
Hamiltonian diffeomorphisms is the kernel of the flux morphism \cite{ban2}.


At the level of Lie algebras of derivations. The algebra $\qInn(M,*)$ is the kernel of the epimorphism
$$F:D\in \der_0(M,*) \mapsto [i(p(D))\omega]\in H^1_{dR}(M)[[\nu]]$$
where we endow $H^1_{dR}(M)[[\nu]]$ with the trivial Lie bracket.

To produce the formal flux morphism we will integrate the morphism $F$ to the group $\aut_0(M,*)$.
For this, we will consider smooth paths in $\aut_0(M,*)$. In the sequel, all the paths considered will be parametrized by $t\in I:=[0,1]$.

All the proofs are similar to the one developped by Banyaga \cite{ban2}, see also \cite{ban}.

Consider $\{A_t\}$ a smooth path in $\aut_0(M,*)$ starting at the identity.  
Set $D_t$ the derivation defined by $\frac{d}{dt}A_t=D_tA_t$.
We set 
\begin{equation} \label{fluxpredef} 
\fl^*(\{A_t\}):=\int_0^1 [i(p(D_t))\omega] dt \in H^1_{dR}(M)[[\nu]].
\end{equation}
We decorate the flux by a $*$ to recall the underlying star product.
We say that two paths $\{A_t\}$ and $\{A'_t\}$ are homotopic with fixed endpoints if there exists a smooth 
map $A_{..}: I\times I \rightarrow \aut_0(M,*): (s,t) \mapsto A_{ts}$ such that $A_{t0}=A_t$ and $A_{t1}=A'_t$ for all $t$, $A_{0s}=A_0$ and $A_{1s}=A_1$ for all $s$. 

\begin{prop} \label{homotdep} 
$\fl^*(\{A_t\})$ only depends on the homotopy class with fixed endpoints of the path $\{A_t\}$.
\end{prop}

Let $A_{ts}$ be a homotopy with fixed endpoints of a path $\{A_t\}$ starting at the identity. There is 
two different ways to define a derivation : 
\begin{equation*}
D_{ts}  :=  (\frac{d}{dt}A_{ts}) \circ A_{ts}^{-1} \textrm{ and }
\tilde{D}_{ts}  :=  (\frac{d}{ds}A_{ts}) \circ A_{ts}^{-1}.
\end{equation*}

\begin{lemme} \label{Dts} 
$\frac{d}{ds}D_{ts}= \frac{d}{dt}\tilde{D}_{ts} + [\tilde{D}_{ts},D_{ts}].$
\end{lemme}

\begin{proof}
Like in the classical case \cite{ban2}, the proof relies on the computations of $\frac{d}{ds}D_{ts}$ and 
$\frac{d}{dt}\tilde{D}_{ts}$ using point $2$ of Proposition \ref{prop:rules}.
\end{proof}

\begin{proof}[Proof of Proposition \ref{homotdep}]
Consider $\{A_{ts}\}$ a homotopy of paths with fixed endpoints. Then, for each $s$, we can compute the flux of the 
path $\{A_{ts}\}$. We show that $\fl^*(A_{t0})=\fl^*(A_{t1})$.
\begin{eqnarray*}
\frac{d}{ds} \fl^*(A_{ts}) & = & \int_0^1 [i(p(\frac{d}{ds}D_{ts}))\omega] dt, \\
& = & \int_0^1 [i(p(\frac{d}{dt}\tilde{D}_{ts}))\omega] dt + \int_0^1[ i(p([\tilde{D}_{ts},D_{ts}]))\omega] dt. \\
& = & [i(p(\tilde{D}_{1s}))\omega - i(p(\tilde{D}_{0s}))\omega].
\end{eqnarray*}
Since the homotopy is with fixed endpoints, $\tilde{D}_{1s}$ and $\tilde{D}_{0s}$ vanishes.
It means that the $\fl^*(A_{ts})$ does not depend on $s$. Then $\fl^*(A_{t0})=\fl^*(A_{t1})$.
\end{proof}

Define $\widetilde{\aut}_0(M,*)$ to be the set of smooth homotopy classes with fixed endpoints of smooth paths $A_t$ of 
automorphisms of the star product starting at the identity. The group structure on $\widetilde{\aut}_0(M,*)$ 
is defined as follows. Let $\{A_t\}$ and $\{B_t\} \in \widetilde{\aut}_0(M,*)$, we set $\{A_t\}.\{B_t\}:= \{A_tB_t\}$.

\begin{theorem} \label{theor:formalflux}
The map
\begin{equation}
\fl^* :\widetilde{\aut}_0(M,*) \rightarrow H^1_{dR}(M)[[\nu]] : \{A_t\} \mapsto \fl^*(\{A_t\})
\end{equation}
is a surjective group morphism.
\end{theorem}

\begin{proof}
By Proposition \ref{homotdep}, the map $\fl^*$ is well defined on $\widetilde{\aut}_0(M,*)$.

We prove that $\fl^*$ is a group morphism. Let $\{A_t\}$ and $\{B_t\} \in\widetilde{\aut}_0(M,*)$
generated by $D_t$ and $D'_t$ respectively. Then the path $\{A_tB_t\}$ is generated by the path $D_t + A_tD'_t(A_t)^{-1}$, by the computation rules \ref{prop:rules}.
Again by Proposition \ref{prop:rules}, $A_tD'_t(A_t)^{-1}=D'_t+D_{F_t}$ for some $F_t \in C^{\infty}(M)[[\nu]]$.
Moreover, $D_{F_t}$ is a smooth family in $\qInn(M,*)$.

Now, we compute the flux.
\begin{eqnarray*}
\fl^*(\{A_tB_t\}) & = & \int_0^1 [ i(p(D_t+ A_tD'_t(A_t)^{-1}))\omega] dt \\
& = & \int_0^1 [i(p(D_t))\omega] dt +\int_0^1 [i(p(D'_t))\omega] dt +\int_0^1[ i(p(D_{F_t}))\omega] dt \\
& = & \fl^*(\{A_t\}) + \fl^*(\{B_t\}).
\end{eqnarray*} 
\end{proof}

\noindent We can now characterize Hamiltonian automorphisms using the formal flux morphism.

\begin{theorem} \label{theor:caractham}
Let $A \in \aut_0(M,*)$.

Then $A \in \ham(M,*)$ if and only if there exists a smooth path $A_t$ of automorphisms,
with $A_0=Id$ and $A_1=A$, such that $\fl^*(\{ A_t\})=0$.

Moreover, the path $A_t$ can be homotoped with fixed endpoints to a path of the form $A_t^H$ generated by some
$H_t \in C^{\infty}(M)[[\nu]]$.
\end{theorem}

\begin{proof}
Assume $A\in \ham(M,*)$. Then  $A=A_1^H$ for some smooth family $D_{H_t} \in \qInn(M,*)$. Then, $\fl^*(\{A_t^H\})= \int_0^1 [i(p(D_{H_t})) \omega ]dt=0$, 
because $p(D_{H_t})$ is a Hamiltonian vector field.

Conversely, assume there exists a smooth path $\{A_t\}$ of automorphisms connecting the identity to $A$ which has 
vanishing $\fl$. This means that there exists a series $F \in C^{\infty}(M)[[\nu]]$ such that $ \int_0^1 [i(p(D_t))\omega] dt = [dF].$
We want to prove that $A\in \ham(M,*)$.

We first observe that we can assume that $\int_0^1 i(p(D_t))\omega dt=0$. Indeed, consider the path of 
automorphisms $C_t:= A_{t} A_t^{-(A_t)^{-1}F}$. Then $\{C_t\}$ is generated by $D_t-D_{F}$ and 
$\int_0^1i(p(D_t-D_F))\omega=0$. Now, since $\ham(M,*)$ is a group, it is sufficient to prove the theorem for $C_1$.

So, suppose our smooth path $\{A_t\}$ satisfies $\int_0^1 i(p(D_t))\omega dt=0$.
Define the family of derivations $D'_t:=-\int_0^t D_u du$. For each $t$, it generates a one-parameter 
group of automorphisms $Q_t^s$, such that $\frac{d}{ds} Q_t^s:=D'_t Q_t^s$. Remark that, since $D'_0=D'_1=0$,
we get $Q_0^s=Q_1^s=Id$. It implies that $A_{ts}:=Q_t^sA_t$ is a homotopy of path with fixed endpoints.
We conclude by showing that $A_{t1}=Q_t^1A_t$ is generated by some series of functions.
We compute
\begin{eqnarray*}
\fl^*(\{A_{t1}\}_{0\leq t \leq T}) & = & \fl^*(\{Q^1_t\}_{0\leq t \leq T}) + \fl^*(\{A_{t}\}_{0\leq t \leq T}) \\
& = & \fl^*(\{Q^s_T\}_{0\leq s \leq 1}) + \fl^*(\{A_{t}\}_{0\leq t \leq T}) \\
& = & \int_0^1[i(p(D'_T))\omega] dt + \int_0^T [i(p(D_t))\omega] dt=0 \\
\end{eqnarray*} 
So, if we write $\overline{D}_t$ the derivation generating the path $A_{t1}$, then we have proved that $\int_0^T i(p(\overline{D}_t)) \omega dt=dF_T$, for all $T\in [0,1]$. Then, 
$i(p(\overline{D}_t))\omega = d(\frac{d}{dt}F_t)$. So, $\overline{D}_t$
is a quasi-inner derivation, which means that $A_{t1}= A_t^{G}$ for the family $G_t:=\frac{d}{dt}F_t$.
This finishes the proof.
\end{proof}

The above Theorem \ref{theor:caractham} implies that $\fl^*$ descends to a morphism
on $\aut_0(M,*)$ whose kernel is $\ham(M,*)$, as we have stated in Theorem \ref{theor:SES}.

\begin{theoremprinc}
There is a short exact sequence of groups
$$1\rightarrow \ham(M,*) \rightarrow \aut_0(M,*) \stackrel{\Fcal}{\rightarrow} \frac{H^1_{dR}(M)[[\nu]]}{\Gamma(M,*)} \rightarrow 1,$$
where $\Fcal(A):=\fl^*(\{A_t\})$ for any smooth path in $\aut_0(M,*)$ joining $A$ to the identity and $\Gamma(M,*):=\fl^*(\pi_1(\aut_0(M,*)))$ where $\pi_1(\aut_0(M,*))$ is the subgroup of $\widetilde{\aut}_0(M,*)$ consisting of classes of smooth loops of automorphisms.
\end{theoremprinc}

\begin{proof}
First, let us check the map $\mathcal{F}$ is well-defined.
Let $A\in \aut_0(M,*)$ and consider $A_t$ and $A'_t$ two smooth paths joining $A$ to $Id$.
Then, the classes $\{A_t\}$ and $\{A'_t\}$ differ  from an element in $\pi_1(\aut_0(M,*))$.
Hence, 
$$\fl^*(\{A_t\})-\fl^*(\{A'_t\}) \in \fl^*(\pi_1(\aut_0(M,*))).$$
Because we quotiented $H^1_{dR}(M)[[\nu]]$ by $\Gamma(M,*)$, the map $\mathcal{F}$ is well defined.

The map $\mathcal{F}$ is a morphism because $\fl^*$ is a morphism.

It remains to verify that $\ham(M,*)$ is the kernel of $\mathcal{F}$.
Clearly, if $A \in \ham(M,*)$, then $\mathcal{F}(A)=0$.
Now, suppose $\mathcal{F}(A)=0$. By definition, when we take a smooth path
connecting $A$ to $Id$, we have
$$\fl^*(\{A_t\})\in \fl^*(\pi_1(\aut_0(M,*))).$$
Then, one can choose a loop $\{B_t\}\in \pi_1(\aut_0(M,*))$ so that
$$\fl^*(\{A_t\})=\fl^*(\{B_t\}).$$
Now, this means that $\fl^*$ vanishes on the path $\{A_tB_t^{-1}\}$. Then, 
by Theorem \ref{theor:caractham} above, it means that its extremity $A$ is a Hamiltonian automorphisms.
\end{proof}

We give a nice geometric interpretation of the group $\Gamma(M,*)$.
There is a natural way to lift a loop in $\symp_0(M,\omega)$ into a path $B_t$ (\emph{not} necessarily a loop)  of automorphisms of the star product.  Elements in $\Gamma(M,*)$ can be used to measure what is needed to close the path $B_t$ into a loop.

Consider a loop $\varphi_t\in \symp_0(M,\omega)$ generated by the smooth time dependent symplectic vector field
$X_t$. Then, consider the unique solution $B_t^{-1}$ of the equation
\begin{equation} \label{eq:Btq}
\frac{d}{dt}B_t^{-1}=-p^{-1}(X_t)B_t^{-1}, \textrm{ with } B_0^{-1}=Id.
\end{equation}
Now, the path $B_t$ is a lift of $\varphi_t$ in the sense that $\cl(B_t)=\varphi_t$.

Since $\varphi_t$ is a loop, $B_1=\exp(D)$ for some $D\in \nu \der_0(M,*)$. Then the above
path $B_t$ can be closed into the loop $\exp(-tD)B_t$. Because $\cl^{-1}(Id)$ is in bijection
with the vector space $\nu\der_0(M,*)$, there is a well-defined isomorphism
$$q:\{\varphi_t\} \in \pi_1(\symp_0(M,\omega))\mapsto \{\exp(-tD)B_t\} \in \pi_1(\aut_0(M,*)).$$

The formal flux morphism induces a morphism
$$\fl^*_{def}:\pi_1(\symp_0(M,\omega)) \rightarrow H^1_{dR}(M)[[\nu]]:\{\varphi_t\} \mapsto \fl^*(q(\{\varphi_t\})).$$
We call $\fl^*_{def}$ the {\bf deformed flux morphism}.
Its image is $\Gamma(M,*)$, because the map $q$ is an isomorphism. 

\begin{prop} \label{prop:fldef}
If $\varphi_t$ is a loop in $\symp_0(M,\omega)$ generated by $X_t$, write $B_t^{-1}$ the solution of equation (\ref{eq:Btq}) and
$D\in \nu \der_0(M,*)$ such that $B_1=\exp(D)$, then
\begin{equation} \label{eq:fldef}
\fl^*_{def}(\{\varphi_t\})=\int_0^1 [i(X_t)\omega]dt - [i(p(D))\omega].
\end{equation}

If $*$ and $*'$ are two equivalent star products, then $\fl^*_{def}(\{\varphi_t\})=\fl^{*'}(\{\varphi_t\})$ for all
$\{\varphi_t\} \in \pi_1(\symp_0(M,\omega))$.
\end{prop}

\begin{proof}
By construction $q(\{\varphi_t\})=\{\exp(-tD)B_t\}$.
Now, we can compute
\begin{eqnarray*}
\fl^*_{def}(\{\varphi_t\}) & = & \fl^*(\{\exp(-tD)B_t\}) \\
& = & \fl^*(\{B_t\})+ \fl^*(\{\exp(-tD)\}) \\
& = & \int_0^1 [i(X_t)\omega]dt - [i(p(D))\omega]. \\
\end{eqnarray*}

Now, let $*$ and $*'$ be two equivalent star products on $C^{\infty}(M)[[\nu]]$. Consider 
$T=Id + \sum_{r=1}^{\infty}\nu^rT_r : (C^{\infty}(M)[[\nu]],*)\rightarrow (C^{\infty}(M)[[\nu]],*')$
an equivalence of star product. We want to prove 
\begin{equation} \label{eq:fldef=}
\fl_{def}^*(\{\varphi_t\})=\fl_{def}^{*'}(\{\varphi_t\}),
\end{equation}
for all $\{\varphi_t\} \in \pi_1(\symp_0(M,\omega))$. To do that, we will decorate by a $'$ all the objects
consider in the hypothesis but corresponding to the star product $*'$. We will use the bijection $p'$ between
derivations of $*'$ and series of symplectic vector fields. We denote by $(B'_t)^{-1}$ the path generated by
$-(p')^{-1}(X_t)$. We set $D'$ the derivation of $*'$ such that $B'_1=\exp(D')$. To prove
equation $(\ref{eq:fldef=})$, we will show $p'(D')=p(D)+X_F$ for some $F\in C^{\infty}(M)[[\nu]]$.

First, we check that $TB_tT^{-1}=B'_t\exp(D'_K)$ for some $D_K\in \nu \qInn(M,*)$.
For this, consider a good cover $\mathcal{U}$ on $M$. On $U\in \mathcal{U}$, $p^{-1}(X_t)\vert_U=\frac{1}{\nu}[H_t^U,.]_*$ 
and $(p')^{-1}(X_t)\vert_U=\frac{1}{\nu}[H_t^U,.]_{*'}$, for some $H_t^U\in C^{\infty}(U)[[\nu]]$. Then we compute
\begin{equation}
Tp^{-1}(X_t)T^{-1}\vert_U= \frac{1}{\nu}[H_t^U,.]_{*'}+\frac{1}{\nu}[\sum_{r=1}^{\infty}\nu^rT_r(H_t^U),.]_{*'}.
\end{equation}
Now, the function $\tilde{H}$ defined by $\tilde{H}\vert_U:=\sum_{r=1}^{\infty}\nu^rT_r(H_t^U)$ for all $U\in \mathcal{U}$ is globally defined. 
So that, we have
$$TB_tT^{-1}=B'_tA_t^{-B'\tilde{H}},$$
because the two paths are generated by the same family of derivations. This means $TB_1T^{-1}=B'_1\exp(D'_K)$
for $K=-\int_0^1 B'_t\tilde{H}_t dt \in \nu C^{\infty}(M)[[\nu]].$

Now, by definition $TB_1T^{-1}=\exp(TDT^{-1})$. Writing locally on $U\in \mathcal{U}$ the derivation $TDT^{-1}$,
we get $p'(TDT^{-1})=p(D)+X_G$. Since,
$$(B'_1)^{-1}=\exp(-T^{-1}DT) \exp(D'_K),$$
by applying the computation rules \ref{prop:rules}, we obtain $p'(D')=p(D)+X_G+X_K+X_{\tilde{K}}$ for some $\tilde{K}\in C^{\infty}(M)[[\nu]]$.
The proof is over.
\end{proof}

\begin{ex} \label{ex:surf}
Consider a symplectic surface $(\Sigma_g,\omega)$ of genus $g\geq 2$. Then one knows $\pi_1(\symp_0(\Sigma_g,\omega))=\{0\}$, see \cite{polte}. Consequently, for all star product $*$ on $(\Sigma_g,\omega)$, we have $\Gamma(\Sigma_g,*)=\{0\}$. The exact sequence of Theorem \ref{theor:SES} writes 
$$1\rightarrow \ham(\Sigma_g,*) \rightarrow \aut_0(\Sigma_g,*) \stackrel{\Fcal}{\rightarrow} \R^{2g}[[\nu]]\rightarrow 1.$$
\end{ex}

\begin{prop} \label{prop:countflux}
The set of de Rham classes arising at order $0$ and $1$ in $\nu$ in elements of $\Gamma(M,*)$ is at most countable.
\end{prop}

\begin{proof}
For our computations, we select a particular star product $*$ on a given equivalence class. Let $\Omega= \nu \Omega_1 + \nu^2 \ldots \in \Omega^2(M)[[\nu]]$ a series of closed $2$-forms that represents the characteristic class parametrizing $*$. Up to equivalence, we can assume that $C_1(.,.)=\frac{1}{2}\{.,.\}$ and
$C_2^-(F,G):= C_2(F,G)-C_2(G,F)=-\Omega_1(X_F,X_G)$.
 
Let $\{\varphi_t\}$ be a path of symplectomorphisms generated by $X_t\in \Symp(M,\omega)$.
We compute its deformed flux at order $1$ in $\nu$.
For this we consider the path $B_t^{-1}$ of automorphisms generated by $-p^{-1}(X_t)$.

We compute 
\begin{equation*}
B_1^{-1}(H)=H + \nu \int_0^1 \varphi_t^*(\Omega_1(X_{t},X_{\varphi_t^{-1*}H}))dt + \nu^2(\ldots)
\end{equation*}
So that, $\fl^*_{def}(\{\varphi_t\})=\int_0^1[i(X_t)\omega]dt + \nu[i(Y_1)\omega] + \nu^2(\ldots)$ for the symplectic vector field $Y_1=\int_0^1 \varphi_t^*(\Omega_1(X_{t},X_{\varphi_t^{-1*}.}))dt$. Moreover, we have
$
i(Y_1)\omega=-\int_0^1 \varphi_t^*(i(X_{t})\Omega_1)dt.
$
Then, we conclude
\begin{equation}
\fl^*_{def}(\{\varphi_t\})= \int_0^1[i(X_t)\omega]dt - \nu \left[\int_0^1 \varphi_t^*(i(X_{t})\Omega_1)dt\right]+\nu^2(\ldots).
\end{equation}

So that, the set of the Rham classes arising at order $0$ and $1$ in $\nu$ of elements of $\Gamma(M,*)$ is at most countable.
\end{proof}

\begin{rem}
The study of $\pi_1(\symp_0(M,\omega))$ through a lifting procedure to loops of automorphisms of star product was also suggested in \cite{miya}.
\end{rem}


\section{Paths of Hamiltonian automorphisms} \label{sect:hampath}

\begin{defi}
The {\bf Lie algebra $\Ham(M,*)$ of $\ham(M,*)$} is the set of derivations $D$ of $*$ such that there exists a smooth path $A:]-\epsilon,\epsilon[\rightarrow \aut_0(M,*)$ for $\epsilon \in \R$ such that $A_t\in \ham(M,*)$ for all $t\in ]-\epsilon,\epsilon[$ and $\frac{d}{dt}|_0A_t=D$, with Lie bracket given by the commutator of derivations.
\end{defi}

One checks $(\Ham(M,*),[,])$ is indeed a Lie algebra, using the computation rules of Proposition \ref{prop:rules}.

\begin{quest} \label{question:algLie}
Is it true that $(\Ham(M,*),[.,.])\cong (\qInn(M,*),[.,.])$?
\end{quest}

By construction, the algebra $\Ham(M,*)$ contains the algebra $\qInn(M,*)$. 
When $H^1_{dR}(M)=0$, a derivation is always of the form $D_H$ for $H\in C^{\infty}(M)[[\nu]]$. Then, $(\Ham(M,*),[.,.])\cong (\qInn(M,*),[.,.])$. However when $H^1_{dR}(M)\neq 0$, we will see the answer is not trivial and depends on the image of $\fl^*_{def}$.

The above question \ref{question:algLie} is equivalent to  the following question : 

\begin{quest} \label{Q1} 
Is any smooth path $A_t\in \ham(M,*)$ generated by a time-dependent Hamiltonian $H_t\in C^{\infty}(M)[[\nu]]$?
\end{quest}

In the classical case, Banyaga \cite{ban2} shows that every path in $\ham(M,\omega)$ is generated by a time dependent Hamiltonian $H_t \in C^{\infty}(M)$.

\begin{theorem} \label{theor:hampath}
Any smooth paths of Hamiltonian automorphisms is generated by a Hamiltonian $H_t\in C^{\infty}(M)[[\nu]]$
if there is no non constant smooth paths in $\fl_{def}(\pi_1(\ham(M,\omega)))\subset H^1_{dR}(M)[[\nu]]$, where $\pi_1(\ham(M,\omega))$ is viewed as a subgroup of $\pi_1(\symp_0(M,\omega))$ using the canonical inclusion.
\end{theorem}

\begin{rem}
In the above Theorem \ref{theor:hampath}, a path in $H^1_{dR}(M)[[\nu]]$ is smooth if and only if its coefficients are smooth paths.
\end{rem}

\begin{proof}
Because $\ham(M,*)$ and  $\fl_{def}(\pi_1(\ham(M,\omega)))$ are groups, it is enough to consider paths starting at the neutral element.

Now, consider $\{A_t\}$ a path of Hamiltonian automorphisms starting at the identity. Then, $\cl(\{A_t\})$ is a path of Hamiltonian diffeomorphisms
which is generated by some $F_t \in C^{\infty}(M)$. So, $\{A_t^{-F}A_t\}$ is a path in $\cl^{-1}(Id)\cap \ham(M,*)$. 
Then it suffices to prove the theorem for paths in $\cl^{-1}(Id)\cap \ham(M,*)$.

Consider a smooth path $\{A_t:=\exp(D_t)\}_{t\in [0,1]}\in \cl^{-1}(Id)\cap \ham(M,*)$ starting at the identity. Then the images of the partial paths $\{A_t\}_{t\in[0,s]}$ for $0\leq s \leq 1$ by $\fl^*$ gives a smooth path $t\mapsto [i(p(D_t))\omega]$ in $H^1_{dR}(M)[[\nu]]$. Because $A_t \in \ham(M,*)$, the path $t\mapsto [i(p(D_t))\omega]$ is in $\fl_{def}(\pi_1(\ham(M,\omega)))$. By hypothesis, the path is contant and $[i(p(D_t))\omega]=[i(p(D_0))\omega]=0$ for all $t$, then $D_t\in \qInn(M,*)$.
\end{proof}

\begin{ex}
Let $(\Sigma_g,\omega)$ be a closed orientable surface of genus $g \geq 1$ equipped with an area form $\omega$. One can show 
$\pi_1(\ham(\Sigma_g,\omega))=0$ for all $g\geq 1$ (see \cite{polte} for an outline of the proof). 
Let $*$ be a star product on $(\Sigma_g,\omega)$.
Then, by Theorem \ref{theor:hampath}, every path $\{A_t\} \in \ham(\Sigma_g,*)$ is generated by a time-dependent Hamiltonian $H_t\in C^{\infty}(\Sigma_g)[[\nu]]$ and then
$$(\Ham(\Sigma_g,*),[.,.])\cong (\qInn(\Sigma_g,*),[.,.]).$$
\end{ex}

\begin{prop} \label{hampath} 
Assume $(M,\omega)$ is a closed symplectic manifold equipped with a star product $*$. Let $A_t$ be a path of Hamiltonian automorphisms, then there exists $H_t:=H^0_t + \nu H^1_t 
\in C^{\infty}_0(M)[[\nu]]$ such that $A_t^{H}=A_t \textrm{ mod }O(\nu^2)$.
\end{prop}

\begin{proof}
In Proposition \ref{prop:countflux}, we showed that at order $0$ and $1$ in $\nu$ the group $\Gamma(M,*)$ is at most countable. Since $\fl_{def}(\pi_1(\ham(M,\omega)))\subset \Gamma(M,*)$, there is no non constant smooth path at order $1$ in $\nu$ in $\fl_{def}(\pi_1(\ham(M,\omega)))$. This is enough to guarantee that paths of Hamiltonian automorphisms are generated by a formal function modulo terms in $O(\nu^2)$.
\end{proof}


\section{Computation using Fedosov's star products} \label{sect:Fed}

In this section we give a nice expression of the deformed flux for nice loops of symplectomorphisms. 
In order to make concrete computation we will use a Fedosov's star product. This is not a restriction in view of Proposition \ref{prop:fldef}.

Let $\Omega \in \nu \Omega^2(M)[[\nu]]$ a formal serie of closed 2-forms and 
$\nabla$ a symplectic connection on $(M,\omega)$ (i.e. a torsion free connection such that $\nabla \omega=0$). 
Through this section we will denote by $*_{\Omega,\nabla}$ the star product obtained via the Fedosov 
construction with respect to $\Omega$ and $\nabla$.

\begin{theorem} \label{fl comp} 
Let $\{\varphi_t\}$ be a loop of symplectomorphisms generated by the symplectic vector field
$X_t$ such that $\varphi_t^*\Omega=\Omega$ and $\varphi_{t*}\nabla=\nabla$ for all $t$.
Then, the deformed flux of $\{\varphi_t\}$ defined with the star product $*_{\Omega,\nabla}$ is
\begin{equation} \label{eq:fluxfed}
\fl_{def}^{*_{\Omega,\nabla}}(\{\varphi_t\})=\int_0^1[i(X_t) \omega]dt - \left[\int_0^1 \varphi^*_t i(X_{t}) \Omega dt\right].
\end{equation}
\end{theorem}


\subsection{Fedosov construction of star product}

We recall the Fedosov construction $\cite{fed2}$, $\cite{fed}$. 
This construction of star product is obtained by identifying $C^{\infty}(M)[[\nu]]$ with the algebra of flat sections of
the Weyl bundle $\W$ endowed with a flat connection.

The sections of the Weyl bundle are formal series of the form :
\begin{equation*}
a(x,y,\nu):=\sum_{2k+l\geq 0} \nu^k a_{k,i_1\ldots i_l}(x)y^{i_1}\ldots y^{i_l}.
\end{equation*}
The $a_{k,i_1\ldots i_l}(x)$ are, in the indices $i_1,\ldots,i_l$, the components of a symmetric tensor on $M$ and $2k+l$ is the degree
in $\W$. The space of sections of $\W$, denoted by $\Gamma \W$, has a structure of an algebra
defined by the fiberwise product
\begin{equation*}
(a \circ b)(x,y,\nu):= \Big( Exp(\frac{\nu}{2}\Lambda^{ij} \partial_{y^i} \partial_{z^j})a(x,y,\nu)b(x,z,\nu)\Big)|_{y=z}
\end{equation*}

To describe connections on $\W$, we will consider forms with values in the Weyl algebra. Those can
be written in local coordinates as
\begin{equation*}
\sum_{2k+l\geq 0} \nu^k a_{k,i_1\ldots i_l,j_1\ldots j_k}(x)y^{i_1}\ldots y^{i_l}dx^{j_1}\wedge \ldots \wedge dx^{j_k}.
\end{equation*}
The $a_{k,i_1\ldots i_l,j_1\ldots j_k}(x)$ are, in the indices $i_1,\ldots,i_l,j_1,\ldots,j_k$, the components of a tensor on $M$, symmetric in the 
$i$'s and antisymmetric in the $j$'s. The space of such sections, $\Gamma \W\otimes \Lambda^*(M)$, is 
endowed with a structure of algebra. For $a\otimes \alpha$ and $b\otimes \beta$, we define
$(a\otimes \alpha) \circ (b\otimes \beta) := a\circ b \otimes \alpha\wedge \beta$. The space of $\W$-valued forms inherits
the structure of a graded Lie algebra from the graded commutator $[s,s']:=s\circ s'- (-1)^{q_1q_2}s'\circ s$, where
$s$ is a form of degree $q_1$ and $s'$ of degree $q_2$.

The connection $\partial$ in $\W$ is defined by
\begin{equation*}
\partial a := da + \frac{1}{\nu}[\overline{\Gamma},a] \in \Gamma \W\otimes \Lambda^1M.
\end{equation*}
where $\overline{\Gamma}:=\frac{1}{2}\omega_{lk}\Gamma^k_{ij}y^ly^jdx^i$ with $\Gamma^k_{ij}$ the Christoffel 
symbols of a symplectic connection $\nabla$ on $(M,\omega)$. Of course, the connection $\partial$ extends to
a covariant derivative on all $\Gamma \W \otimes \Lambda M$ using the Leibniz rule :
\begin{equation*}
\partial (a\otimes \alpha) := (\partial a) \wedge \alpha + a\otimes d\alpha.
\end{equation*}

The curvature of $\partial$ is denoted by $\partial \circ \partial$ and is expressed in terms
of the curvature $R$ of the symplectic connection $\nabla$. 
\begin{equation*}
\partial\circ \partial a := \frac{1}{\nu}[\overline{R},a],
\end{equation*}
where $\overline{R}:= \frac{1}{4} \omega_{ir}R^r_{jkl}y^iy^jdx^k\wedge dx^l$.

Define $$\delta(a) := dx_k\wedge \partial_{y_k} a=-\frac{1}{\nu}[\omega_{ij}y^i dx^j,a],$$
and 
$$\delta^{-1} a_{pq}:= \frac{1}{p+q}y^ki(\partial_{x^k})a_{pq} \textrm{ if } p+q>0 \textrm{ and } \delta^{-1}a_{00}=0,$$
where $a_{pq}$ is a $q$-forms with $p$ $y$'s and $p+q>0$.
We then have the Hodge decomposition of $\Gamma \W\otimes \Lambda M$ :
$\delta \delta^{-1}a + \delta^{-1} \delta a=a-a_{00}$.

Now, we recall the construction of a flat connection $\D$ on 
$\Gamma \W$ of the form
\begin{equation*}
\D a:=\partial a - \delta a + \frac{1}{\nu}[r,a],
\end{equation*}
where $r$ is a $\W$-valued $1$-form and $\D^2a=0$.

Because,
\begin{equation*}
\D^2 a = \frac{1}{\nu}\left[\overline{R} + \partial r - \delta r + \frac{1}{2\nu}[r,r],a\right],
\end{equation*}
one choose $r$ such that 
\begin{equation*}
\overline{R} + \partial r - \delta r + \frac{1}{2\nu}[r,r] = \Omega ,
\end{equation*}
for a central $2$-form $\Omega$. 
Which means that $\D r= \Omega - \overline{R} + \frac{1}{2\nu}[r,r]$. Then, for all closed central $2$-form $\Omega$,
there exists a unique solution $r \in \Gamma \W \otimes \Lambda^1 M$ of degree at least $3$
of the equation
\begin{equation*}
r=\delta^{-1}(\overline{R} + \partial r + \frac{1}{\nu}r\circ r - \Omega).
\end{equation*}
satisfying $\delta^{-1}r=0$.

Define $\Gamma \W_{\D} := \{a\in \Gamma \W | \D a=0\}$ the algebra of flat sections and
the symbol map $\sigma :a\in \Gamma \W_{\D} \mapsto a_{00}\in C^{\infty}(M)[[\nu]]$. Fedosov showed
that the symbol map is a bijection on flat sections and that $a \in \Gamma \W_{\D} $ is the 
unique solution of
\begin{equation*}
a=a_{00}+ \delta^{-1}(\partial a + \frac{1}{\nu}[r,a]).
\end{equation*}
Let $Q$ be the inverse of $\sigma$. Fedosov defined a star product $*_{\Omega,\nabla}$ on $(M,\omega)$ by
\begin{equation*}
F*_{\Omega,\nabla}G:=\sigma(QF \circ QG).
\end{equation*}

In the sequel we will need some low degree terms of $QF$ for $F\in C^{\infty}(M)$.
\begin{equation} \label{eq:QFleq3}
QF = F + \partial_iFy^i + \frac{1}{2}(\nabla_i X_F)_jy^i y^j + (QF)^{\geq 3}=:(QF)^{<3} + (QF)^{\geq 3},
\end{equation}
where $(\nabla_i X_F)_j = (\nabla_i X_F)^k\omega_{kj}$ and $(QF)^{\geq 3}$ denotes the term of
degree bigger than $3$.


\subsection{Exponentiation of derivations}

In this subsection, we write the solution of the equation (\ref{eq:Heis}) in the Weyl algebra. 
This is the first step in the proof of Theorem \ref{fl comp}. We follow
the book \cite{fed}.

We first translate the equation (\ref{eq:Heis}) in the Weyl algebra. Let $\U$ be a good cover of $M$, then for all $U\in \U$
there exists a serie $H_t^U:= \sum_{r=0}^{\infty} \nu^r H_{t,i}^U \in C^{\infty}(U)[[\nu]]$ such that $D_t\vert_U=D_{H_t^U}$.
We can then consider the local section $QH_t^U$ of $\W$. Because two functions $H_t^U$ and $H_t^{U'}$ differ on $U\cap U'$
by a constant, we can define a global section
$$(QH_t^{\U}-H_t^{\U})(x):=(QH_t^U-H_t^U)(x) \textrm{ if } x\in U.$$
Then, to solve the equation (\ref{eq:Heis}), we build the unique family $A_t$ of automorphisms of $\Gamma\W_{\D}$
such that for all $s\in \Gamma \W_{\D}$ : 
\begin{equation} \label{eq:HeisFed} 
\frac{d}{dt} A_t(s)=\frac{1}{\nu}[QH_t^{\U}-H_t^{\U},A_t(s)],
\end{equation}
with initial condition $A_0=Id$.

The strategy is the same as in proposition $\ref{flow}$. 

We define the natural pull-back
on $\Gamma \W \otimes \Lambda M$ by a symplectomorphism $\varphi$ :
\begin{eqnarray*}
\varphi_*(a(x,y,\nu)\otimes \alpha)& := & a(\varphi(x),(\varphi_{*x})^{-1}y,\nu)\otimes \varphi^*\alpha \nonumber \\
& = & \sum_{2k+l\geq 0} \nu^k a_{k,i_1\ldots i_l}(x)(\partial_{x^{j_1}}\varphi)^{i_1}\ldots (\partial_{x^{j_l}}\varphi)^{i_l}y^{j_1}\ldots y^{j_l}\otimes \varphi^*\alpha.
\end{eqnarray*}
In the paper \cite{gr3}, Gutt and Rawnsley showed that the Lie derivative satisfies the following Cartan
formula. Let $\varphi_t$ be a symplectic isotopy generated by the time dependent symplectic 
vector field $X_t$, then
\begin{equation} \label{eq:lieder}
\frac{d}{dt}\varphi_{t*}=\varphi_{t*}\Big( i(X_t)\D+\D i(X_t) + \frac{1}{\nu}ad_{\circ}( \omega_{ij}X_t^iy^j + \frac{1}{2}(\nabla_iX_t)_jy^iy^j - i(X_t)r)\Big),
\end{equation}
where $ \frac{1}{2}(\nabla_iX_t)_jy^iy^j:= \frac{1}{2}(\nabla_iX_t)^k\omega_{kj}y^iy^j$.

Now, we solve the equation (\ref{eq:HeisFed}). Consider the symplectic vector field $X_t:=D_{t,0}$ where $D_t=D_{t,0}+\nu(\ldots)$. Consider 
the isotopy $\varphi_t$ whose inverse is generated by $-X_t$. Then, if $A_t$ is a solution of equation (\ref{eq:HeisFed}), using Equation (\ref{eq:lieder}), we have 
\begin{eqnarray} \label{eq:recu} 
\frac{d}{dt} \varphi_{t*}^{-1}A_t(.) & = & \varphi_{t*}^{-1}\frac{-1}{\nu}\left[ \omega_{ij}X_{t}^iy^j + \frac{1}{2}(\nabla_iX_{t})_jy^iy^j - i(X_{t})r,A_t(.)\right] \nonumber \\
& & + \frac{1}{\nu}\varphi_{t*}^{-1}[QH^{\U}_t-H_t^{\U},A_t(.)], \nonumber \\
& = & \frac{1}{\nu} \left[ \varphi_{t*}^{-1}\left((QH_{t,0}^{\U})^{\geq 3}+i(X_{t})r+\tilde{H}_t \right),\varphi_{t*}^{-1}A_t(.)\right],
\end{eqnarray}
where 
$$(QH_{t,0}^{\U})^{\geq 3}(x):=(QH_{t,0}^U)^{\geq 3}(x) \textrm{ and }\tilde{H}_t(x)=\sum_{r\geq 1} \nu^r(QH_{t,r}^{U}-H_{t,r}^U)(x) \textrm{ for } x\in U;$$ 
we recall that $(QH_{t,0}^{U})^{\geq 3}$ denotes the terms of degree bigger than $3$ in $QH_{t,0}^{U}$.
Because $\frac{1}{\nu} \varphi_{t*}^{-1}\left((QH_{t,0}^{\U})^{\geq 3}+i(X_{H_t})r+ \tilde{H}_t\right)$ has degree greater or equal than $1$, 
the solution of the equation $(\ref{eq:recu})$ is obtained by exponentiation.
And,
\begin{equation} \label{eq:authamfedosov}
A_t(.)=\varphi_{t*}\exp\left(\frac{1}{\nu}ad\left(\int_0^t\varphi_{s*}^{-1}\left((QH_{s,0}^{\U})^{\geq 3}+i(X_{s})r + \tilde{H}_s\right)ds\right)\right).
\end{equation}


\subsection{The deformed flux in the Fedosov's construction}

Let $\{\varphi_t\}$ be a loop of symplectomorphisms generated by $X_t$. Let us explain how to compute its deformed
flux in the Fedosov's formalism. 

Choose a good cover $\U$ of $M$.
Then, on $U\in \U$, there exists $H_t^{U}\in C^{\infty}(U)$ such that $X_t\vert_U=X_{H_t^U}$.
Next, we solve the equation
\begin{equation*}
\frac{d}{dt} A_t^{-1}(a)=\frac{1}{\nu}\left[ -(QH^{\mathcal{U}}_t-H^{\mathcal{U}}_t),A_t^{-1}(a)\right],
\end{equation*}
for all $a\in \Gamma \W_{\D}$ with initial condition $A_0^{-1}=Id$. By equation (\ref{eq:authamfedosov}), we have
\begin{equation} \label{eq:A1flux}
A_1= \exp\left(\frac{1}{\nu}ad\left(\int_0^1\varphi_{s*}\left((QH^{\mathcal{U}}_s)^{\geq 3}+i(X_{s})r\right)ds\right)\right).
\end{equation}
Then, to compute $\fl_{def}^*(\{\varphi_t\})$ we have to find a serie $Y\in \nu \Symp(M,\omega)$ which writes locally as $Y|_U=X_{F^U}$ for some $F^U\in \nu C^{\infty}(U)[[\nu]]$ such that
$$A_1=\exp \left(\frac{1}{\nu}ad\left((QF^{\mathcal{U}}-F^{\mathcal{U}}\right)\right).$$

\begin{proof}[Proof of Proposition \ref{fl comp}]
We assume that $\{\varphi_t\}$ preserves $\nabla$ and $\Omega$. Recall that 
$\{ \varphi_t\}$ is generated by $X_t$ and we write locally $X_t\vert_U=X_{H_t^U}$ for
$H_t^U \in C^{\infty}(U)[[\nu]]$. We consider the automorphism $A_1$ defined in Equation (\ref{eq:A1flux}).
The goal is to prove that $\int_0^1\varphi_{t*}\left((QH^{\mathcal{U}}_t)^{\geq 3}+i(X_{t})r\right) dt=QF^{\mathcal{U}}-F^{\mathcal{U}}$ with
$F^U$ satisfying $dF^U=\int_0^1 \varphi^*_t i(X_{t}) \Omega dt|_U$.

We compute $\mathcal{D}\int_0^1 \varphi_{t*}\left((QH^{\mathcal{U}}_t)^{\geq 3}+i(X_{t})r\right)dt$. By assumption,
$\varphi_{t*}\circ \partial=\partial \circ \varphi_{t*}$ and $\varphi_{t*}r = r$.
This imply that $\varphi_{t*}\circ \D = \D \circ \varphi_{t*}$. Thus, we have
\begin{eqnarray}
\mathcal{D}\int_0^1 \varphi_{t*}\left((QH^{\mathcal{U}}_t)^{\geq 3}+i(X_{t})r\right)dt & = & \int_0^1 \varphi_{t*}\mathcal{D}\left((QH^{\mathcal{U}}_t)^{\geq 3}+i(X_{t})r\right)dt \nonumber \\
 & = & \int_0^1 \varphi_{t*}\left(-\mathcal{D}(QH^{\mathcal{U}}_t)^{< 3}+\mathcal{D}i(X_{t})r\right)dt. \nonumber \\
 & & 
\end{eqnarray}
Where we use the fact that $QH_t^{\U}$ is locally a flat section.
Remark that the section $\mathcal{D}(QH^{\mathcal{U}}_t)^{< 3}$ is globally defined.
Since $\varphi_{t}$ preserves $r$, using (\ref{eq:lieder}), we have
\begin{eqnarray*}
\int_0^1 \varphi_{t*}\mathcal{D}i(X_{t})rdt & & \\
& & \kern-1.5in  = \int_0^1-\varphi_{t*}\Bigg( i(X_{t})\mathcal{D}r +  \frac{1}{\nu}\Big[ \omega_{ij}X_{t}^iy^j + \frac{1}{2}(\nabla_iX_{t})_jy^iy^j - i(X_{t})r,r\Big]\Bigg)dt \nonumber \\
& & \kern-1.5in = \int_0^1-\varphi_{t*}\Big( i(X_{t})\Omega-i(X_{t})\overline{R}  
 + \frac{1}{\nu}\Big[ \omega_{ij}X_{t}^iy^j + \frac{1}{2}(\nabla_iX_{t})_jy^iy^j ,r\Big]\Big) dt
\end{eqnarray*}
We have to compare the above expression with 
\begin{equation} \label{eq:fedterm1}
-\int_0^1 \varphi_{t*}\mathcal{D}(QH^{\mathcal{U}}_t)^{< 3} dt = -\int_0^1 \varphi_{t*}\mathcal{D}\left(H^{\mathcal{U}}_t + \omega_{ij}X_{t}^iy^j + \frac{1}{2}(\nabla_iX_{t})_jy^iy^j\right)dt.
\end{equation}
Remark that the right hand side is globally defined.
So, we compute
\begin{equation} \label{eq:fedD1}
\mathcal{D}H^{\mathcal{U}}_t=dH^{\mathcal{U}}_t=\delta (\omega_{ij}X_{t}^iy^j).
\end{equation}
Also, we have
\begin{equation} \label{eq:fedD2}
\partial (\omega_{kj}X_{t}^ky^j)=\delta \left(\frac{1}{2}(\nabla_iX_{t})_jy^iy^j\right).
\end{equation}
And,
\begin{equation} \label{eq:fedD3}
\partial \left( \frac{1}{2}(\nabla_iX_{t})_jy^iy^j\right)=\frac{1}{2} (\nabla^2_{ij}X_{t})^l\omega_{lk}y^ky^jdx^i.
\end{equation}
Using equations $(\ref{eq:fedD1})$ to $(\ref{eq:fedD3})$ the equation $(\ref{eq:fedterm1})$ becomes
\begin{eqnarray*}
\int_0^1 \varphi_{t*}\left(\mathcal{D}(-(QH^{\mathcal{U}}_t)^{< 3}\right) dt & = & -\int_0^1 \varphi_{t*}\Big(\frac{1}{2} (\nabla^2_{ij}X_{t})^l\omega_{lk}y^ky^jdx^i \nonumber \\
& & \kern-0.5in +\frac{1}{\nu} \left[ r, \omega_{ij}X_{t}^iy^j + \frac{1}{2}(\nabla_iX_{t})_jy^iy^j\right]\Big) dt.
\end{eqnarray*}
Then,
\begin{equation*}
\mathcal{D}\int_0^1 \varphi_{t*}((QH^{\mathcal{U}}_t)^{\geq 3}+i(X_{t})r)dt= \int_0^1 \varphi_{t*}(-i(X_{t})\Omega+i(X_{t})\overline{R}-\frac{1}{2} (\nabla^2_{ij}X_{t})_ky^ky^jdx^i).
\end{equation*}
Remark that $i(X_{t})\overline{R}-\frac{1}{2} (\nabla^2_{ij}X_{t})_ky^ky^jdx^i=\frac{1}{2} (\mathcal{L}_{X_{t}}\nabla)_{ijk}y^ky^jdx^i$.
Since $\varphi_t$ preserves the symplectic connection, we get
\begin{equation*}
\mathcal{D}\int_0^1 \varphi_{t*}\left((QH^{\mathcal{U}}_t)^{\geq 3} + i(X_{t})r\right)dt= - \int_0^1 \varphi_{t}^*\left(i(X_{t})\Omega\right) dt.
\end{equation*}
Consequently, if $F^U \in C^{\infty}(U)[[\nu]]$ satisfies $dF^U=\int_0^1 \varphi_{t}^*i(X_{t})\Omega dt|_U$, then
$QF^{\mathcal{U}}-F^{\mathcal{U}} = \int_0^1\varphi_{t*}((QH^{\mathcal{U}}_t)^{\geq 3} + i(X_{t})r)dt)$. 
It means that : 
$$\fl_{def}^{*_{\Omega,\nabla}}(\{\varphi_t\})=\int_0^1\left[ i(X_t) \omega \right]dt - \left[\int_0^1 \varphi_{t}^*i(X_{t})\Omega dt \right].$$
The proof is over.
\end{proof}

\begin{ex}
Consider the $2$-torus  
$(\mathbf{T}^2,d\theta_1\wedge d\theta_2)$ with usual coordinates $(\theta_1,\theta_2)$. 
The group $\pi_1(\symp(\mathbf{T}^2,d\theta_1\wedge d\theta_2))$
is known to be generated by the rotations $\{ \varphi_t\}$ et $\{ \psi_t \}$ along the symplectic vector fields
$\partial_{\theta_1}$ and $\partial_{\theta_2}$, see \cite{polte}.

Consider Fedosov's star products of the form 
$*_{\Omega,d}$ with $d$ the flat connection and $\Omega=\sum_{i=1}^{\infty}\nu^i C_i.d\theta_1\wedge d\theta_2$.
Then, all the equivalence classes of star product are represented.
Because the 2-form $d\theta_1\wedge d\theta_2$ and $d$ are preserved by $\{ \varphi_t\}$ and $\{ \psi_t \}$, 
we can use Theorem \ref{fl comp} to compute the deformed flux. We obtain
\begin{equation*}
\fl_{def}^{*_{\Omega,d}}(\{\varphi_t\})  =  d\theta_2(1 - \sum_{i=1}^{\infty}\nu^i C_i) \textrm{ and }
\fl_{def}^{*_{\Omega,d}}(\{\psi_t\})  =   d\theta_1(1- \sum_{i=1}^{\infty}\nu^i C_i).
\end{equation*}
Then, $\Gamma(M,*_{\Omega,d})=< d\theta_2(1 - \sum_{i=1}^{\infty}\nu^i C_i), d\theta_1(1- \sum_{i=1}^{\infty}\nu^i C_i)>_{\Z}$.\\
The exact sequence of Theorem \ref{theor:SES} writes
$$1\rightarrow \ham(T^2,*_{\Omega,d}) \rightarrow \aut_0(T^2,*_{\Omega,d}) \stackrel{\mathcal{F}}{\rightarrow} \frac{H^1_{dR,(c)}(T^2)[[\nu]]}{(1 - \sum_{i=1}^{\infty}\nu^i C_i)< d\theta_2, d\theta_1>_{\Z}} \rightarrow 1.$$
\end{ex}

\begin{cor}
Consider the symplectic manifold $(\mathbf{T}^2,d\theta_1\wedge d\theta_2)$. \\
Two star products $*$ and $*'$ are equivalent if and only if $\Gamma(\mathbf{T}^2,*)=\Gamma(\mathbf{T}^2,*').$
\end{cor}


We think our results together with the example of symplectic surfaces motivate a deeper study of the groups $\Gamma(M,*)$.
In particular, it would be nice to generalize the formula (\ref{eq:fluxfed}) to arbitrary loops in $\symp_0(M,\omega)$.
It would imply $\Gamma(M,*)$ is at most countable at any order in $\nu$. That would mean the hypothesis
in Theorem $\ref{theor:hampath}$ above is always satisfied.

As a concluding remark we mention that the flux morphism can be defined on certain Poisson manifolds \cite{Xu,Ryb}.
It might be interesting to see how the work of this paper extends to star products on Poisson (non symplectic) manifolds.


\end{document}